\documentclass{amsart}
\usepackage{amssymb,amsmath,amsthm}
\usepackage[foot]{amsaddr}
\title{Initial algebras from constructive ordinals}
\date{\today}
\author{Benno van den Berg}
\address{Institute for Logic, Language and Computation, Universiteit van Amsterdam, P.O. Box 94242, NL, 1090 GE Amsterdam, The Netherlands}
\email{b.vandenberg3@uva.nl}
\usepackage{ast2}
\usepackage{bussproofs}
\usepackage[margin=1.7in]{geometry}

\usepackage{tikz-cd}

\usepackage{float}
\restylefloat{table}

\newcommand{\cleq}{\preccurlyeq}

\newcommand{\inr}{{\sf inr}}

\usepackage[svgnames]{xcolor}

\usepackage[colorlinks=true,
            citecolor=Fuchsia,
            linkcolor=Fuchsia, 
            urlcolor=Fuchsia
            ]{hyperref}

\usepackage{geometry}

\begin{document}

\begin{abstract}
    We show how a standard constructive notion of ordinal supports a useful constructive theory of transfinite recursion. We do this by giving constructive proofs of various initial algebra theorems, like Ad\'amek's theorem, and a version of Quillen's small object argument for constructing cofibrantly generated algebraic weak factorisation systems.
\end{abstract}

\maketitle

\section{Introduction}

The notion of a well-order is one of the fundamental notions of logic, playing a prominent role in set theory, proof theory and recursion theory. Ordinals arise as canonical representatives of isomorphism classes of well-orders. They measure the complexity of inductive definitions or inductively defined objects, as well as the strength of theories in proof theory. 

Ordinals are also naturally thought of as the stages in an iterative process, where they allow us to continue an iterative process beyond any finite stage. This process of \emph{transfinite recursion} is one of the main applications of the notion of ordinal: indeed, this is why Cantor invented them. One typical application of transfinite recursion is the following result due to Ad\'amek on the existence of initial algebras  for endofunctors.

\begin{theo}{adadamek}{\rm (Ad\'amek \cite{adamek74})}
        Let $F$ be an endofunctor on a cocomplete category \ct{C}, and $\lambda$ be a limit ordinal. If $F$ preserves colimits of shape $\lambda$, then $F$ has an initial algebra. 
\end{theo}  

Various variations on this result exist and one such can be used to prove a form of Quillen's small object argument, guaranteeing the existence of cofibrantly generated algebraic weak factorisation systems. The following result can be found in \cite[Theorem 18]{bergetal25}, building on work by Quillen, Bourke and Garner, see \cite{Quillen1967,Garner2011Understanding,Bourke2016Accessible}.

\begin{theo}{smallobjarg}{\rm (Double-categorical small object argument)}
        Let \ct{C} be a locally small and cocomplete category, and $\lambda$ be a limit ordinal. If $\mathbb{A}$ is a small double category over \ct{C} such that $Aa$ is $\lambda$-presentable for each object $a$ in $\mathbb{A}$, then the algebraic weak factorization system cofibrantly generated by $\mathbb{A}$ exists.
\end{theo}

This paper is concerned with applications of transfinite recursion in a constructive metatheory. Indeed, we will show that by employing a suitable constructive notion of ordinal the two results we have just mentioned, including some variations on Ad\'amek's result, can be proven constructively.

We will leave it somewhat vague what we mean by a constructive metatheory; however, all the arguments that we will present in this paper can be performed in Aczel's constructive set theory {\bf CZF} \cite{aczelrathjen01}, so the reader who wishes to work within a fixed metatheory can choose that. They can also be formalised in a suitable version of Martin-L\"of's constructive type theory \cite{martinlof84}. In particular, we will not need or use any form of impredicativity. 

The classical notion of ordinal splits up in various different notions which are no longer equivalent constructively, which means that in a 
constructive setting there are a number of notions of ordinal to choose from (see, for instance, \cite{joyalmoerdijk95,taylor96,krausetal21,coquandetal23,dejongetal23,dejongetal25}). Fortunately, there is a standard notion that works quite well for our purposes. It exists in the following incarnations:
\begin{itemize}
    \item[-] as the notion of a transitive set of transitive sets. This notion, due to Powell, is the standard definition of an ordinal in constructive set theory.
    \item[-] as the notion of well-order in homotopy type theory \cite{univalent13}.
    \item[-] as an element of the initial ZF-algebra with a progressive successor in algebraic set theory (what Joyal and Moerdijk in \cite{joyalmoerdijk95} call a  ``von Neumann ordinal'').
\end{itemize}
In fact, we will give a self-contained presentation of this notion of ordinal in Section 2 of this paper.

This constructive notion of ordinal is classically correct in that in a classical metatheory it is equivalent to the standard definition: indeed, while the presentation may be unusual, all our definitions and proofs in this paper can also be formalised in {\bf ZF} (in particular, we avoid the use of the axiom of choice). At the same time, this constructive notion of ordinal cannot be assumed to possess a number of features the classical set theorist would expect them to have. We mention three principles that should be avoided for these constructive ordinals.
\begin{itemize}
    \item Any inhabited class of ordinals has a least element (``least ordinal principle'')
    \item Ordinals are linearly ordered.
    \item An ordinal is either zero, a successor ordinal or a limit ordinal.
\end{itemize}
Assuming that our ordinals satisfy any of these principles implies the Principle of Excluded Middle. For that reason, we will proceed in an agnostic fashion: we will neither assume that these are principles are true, nor will we assume that they are false. Clearly, this means that a number of standard proof methods have to be rethought. The goal of this paper is to demonstrate that this is possible and constructive proofs of the theorems mentioned above exist.

Specifically, the contents of this paper are as follows. As mentioned, we will give a self-contained account of our theory of ordinals in Section 2. In Section 3 we will give a constructive proof of Ad\'amek's theorem, \reftheo{adadamek}. It turns out that with some minor differences in presentation this was already constructively proven by Pitts and Steenkamp \cite{pittssteenkamp21}. Nevertheless, we have decided to include a detailed proof here, because two other main results in this paper (\reftheo{convergence} and \reftheo{specialconvergence}) build on this argument. Sections 4 and 5 will be devoted to proofs of these results. In Section 6 we will give a constructive proof of the small object argument, \reftheo{smallobjarg}. We end this paper with a discussion of some open questions and directions for future research, as well as two appendices. The first appendix discusses in detail the relationship of our treatment of ordinals with Joyal and Moerdijk's work on algebraic set theory \cite{joyalmoerdijk95}, while the second contains a verification of some technical conditions required for applying Beck's monadicity theorem.

\section{Our theory of ordinals}

To get started, we give a self-contained presentation of our theory of ordinals. Instead of defining what an ordinal is, we will define the class of all ordinals as the result of a certain inductive process.

The way to build an ordinal is to take a set of them and use them to create a new one. If $A$ is a set of ordinals, we will denote this new ordinal as $\sup(A)$; we can think of the elements of $A$ as the \emph{parents} of the ordinal $\sup(A)$, which is therefore a \emph{child} of all the elements in $A$. In other words, we have an operation
\begin{eqnarray*} 
    \sup & : & {\rm Pow}({\rm Ord}) \to {\rm Ord}
\end{eqnarray*}
(Here ${\rm Pow}(X)$ stands for the class of all subsets of a class $X$: it will not be assumed to be a set, even when $X$ is.) Since the ordinals are inductively generated by this operation, they satisfy the following induction principle.
\begin{displaymath}
    \begin{array}{cc} 
        ({\rm Induction}) & \frac{(\forall A) \, \big( \, (\forall \alpha \in A) \, \varphi(\alpha)  \to \varphi(\sup A) \, \big)}{(\forall \alpha) \, \varphi(\alpha) } 
    \end{array}
\end{displaymath}
The ordinals involved in the creation of an ordinal, either as a parent or as an ordinal involved in the creation of one of its parents, are called its \emph{ancestors}. The operation
\begin{eqnarray*} 
    \Downarrow & : & {\rm Ord} \to {\rm Pow}({\rm Ord}),
\end{eqnarray*}
assigns to each ordinal its set of ancestors and is defined by recursion as follows:
\begin{displaymath}
        \begin{array}{cc}
        ({\rm Ancestors}) & \Downarrow (\sup A) = A \cup \bigcup_{\alpha \in A} \Downarrow \alpha
        \end{array}
    \end{displaymath}
Finally, we will identify ordinals if they have the same ancestors:
    \begin{displaymath}
        \begin{array}{cc}
        ({\rm Extensionality}) &  \frac{\Downarrow \alpha = \Downarrow \beta}{\alpha = \beta}
        \end{array}
    \end{displaymath}
This completes our theory of ordinals. 

It will be convenient to introduce the following notation:
\begin{eqnarray*} \beta \lt \alpha 
    & :\Leftrightarrow &  \beta \in \Downarrow \alpha \\
    \beta \leq \alpha & :\Leftrightarrow & \beta \lt \alpha \lor \beta = \alpha.
\end{eqnarray*}
So $\beta \lt \alpha$ means that $\beta$ is an ancestor of $\alpha$. Using this notation, the recursive definition describing ancestors of ordinals of the form $\sup A$ can be restated as follows:
\begin{equation} \label{usefuleq} \begin{array}{ccccc} \beta \lt \sup A & \Leftrightarrow & \beta \in A \lor (\exists \alpha \in A) \, \beta \lt \alpha
    & \Leftrightarrow & (\exists \alpha \in A) \, \beta \leq \alpha. 
\end{array}
\end{equation}

\begin{rema}{compwithvonneumann}
    Following Von Neumann, in both {\bf CZF} or {\bf ZF} an ordinal is identified with the set of its ancestors (turning $\Downarrow$ into the identity), which means that $\beta \lt \alpha$ holds precisely when $\beta \in \alpha$, and $\beta \leq \alpha$ holds precisely when $\beta \in \alpha$ or $\beta = \alpha$. We will not make this identification here. It should be noted that in a constructive setting the relation $\beta \leq \alpha$ is stronger than $\Downarrow \beta \subseteq \Downarrow \alpha$ (or $\beta \subseteq \alpha$ if you would identify ordinals with the set of its ancestors). The weaker partial order is taken as basic in algebraic set theory (see Appendix A where it is denoted $\cleq$), but, remarkably, plays no role in the main arguments of this paper.
\end{rema}

\begin{lemm}{propoflt}
    \begin{enumerate}
        \item[(1)] $\lt$ is transitive.
        \item[(2)] The following $\lt$-induction scheme is valid:
        \begin{displaymath}
            \frac{(\forall \alpha) \, \big( \, (\forall \beta \lt \alpha) \, \psi(\beta) \to \psi(\alpha) \, \big)}{(\forall \alpha) \, \psi(\alpha)}
        \end{displaymath}   
        \item[(3)] $\alpha \lt \alpha$ never holds.
        \item[(4)] $\leq$ is a partial order.
        \item[(5)] $\sup( \Downarrow \alpha) = \alpha$. 
    \end{enumerate}
\end{lemm}
\begin{proof} (1) We show $\gamma \lt \beta \lt \alpha \Rightarrow \gamma\ \lt \alpha$ by induction $\alpha$. So assume $\gamma \lt \beta \lt \sup A$ and the desired statement holds for all $\alpha \in A$. Since $\beta \lt \sup A$, we have $\beta \in A$ or $\beta \lt \alpha$ for some $\alpha \in A$. In both cases we obtain $\gamma \lt \sup A$: in the former case because $\gamma \lt \beta$ and $\beta \in A$; in the latter case because we have $\alpha \in A$ and $\gamma \lt \alpha$ by induction hypothesis.

(2) We assume $(\forall \alpha) \, \big( \, (\forall \beta \lt \alpha) \, \psi(\beta) \to \psi(\alpha) \, \big)$ and prove $(\forall \beta \leq \alpha) \, \psi(\beta)$ by induction on $\alpha$. So let $A$ be a set of ordinals such that $(\forall \beta \leq \alpha) \, \psi(\beta)$ holds for all ordinals $\alpha \in A$; our aim is to prove $\psi(\sup A)$. For this purpose it suffices to prove $\psi(\beta)$ for all $\beta \lt \sup A$. But the latter means that $\beta \leq \alpha$ for some $\alpha \in A$. So we get $\psi(\beta)$ directly from our assumption.

(3) We prove this by induction on $\alpha$. So assume $\alpha \lt \alpha$ is false for all $\alpha \in A$ and $\sup A \lt \sup A$ holds. The latter means that there is some $\alpha \in A$ such that $\sup A \leq \alpha$ holds. Since $\alpha \lt \sup A$, we obtain a contradiction.

(4) follows directly from (1) and (3).

(5) In view of extensionality it suffices to prove $\Downarrow \big( \sup (\Downarrow \alpha) \big) = \Downarrow \alpha$ for all ordinals $\alpha$. This means we need to show that
\[ (\exists \gamma \lt \alpha) \, \beta \leq \gamma \Leftrightarrow \beta \lt \alpha. \]
But this follows directly from transitivity of $\lt$.
\end{proof}

\begin{rema}{ordinalsasposets}
    In what follows we will always consider $\Downarrow \alpha$ as a poset (and hence a category) with $\leq$ as its ordering.
\end{rema}    

It is sometimes convenient to consider $\sup$ as the composition of the operations $s: {\rm Ord} \to {\rm Ord}$ and $\bigcurlyvee: {\rm Pow}({\rm Ord}) \to {\rm Ord}$ from Algebraic Set Theory defined by:
\begin{eqnarray*}
    s\alpha & := & \sup \{ \alpha \} \\
    \bigcurlyvee A & := & \sup \big( \bigcup_{ \alpha \in A} \Downarrow \alpha \, \big)
\end{eqnarray*}
We will take a closer look at these operations in Appendix A; for now, it suffices to observe the following.
\begin{lemm}{decompositionofsup}
\begin{enumerate}
    \item[(1)] $\beta \lt s\alpha \Leftrightarrow \beta \leq \alpha$.
    \item[(2)] $\beta \lt \bigcurlyvee A \Leftrightarrow (\exists \alpha \in A) \, \beta \lt \alpha$.
    \item[(3)] $\sup A = \bigcurlyvee_{\alpha \in A} s\alpha$. 
    \item[(4)] $\alpha = \bigcurlyvee_{\beta \lt \alpha} s\beta$. 
\end{enumerate}
\end{lemm}
\begin{proof}
    Item (1) follows directly from the equivalence in (\ref{usefuleq}), while item (2) also uses transitivity of $\lt$. To prove (3) we use extensionality, the same equivalence and items (1) and (2). Item (4) is the special case of (3) for $A := \Downarrow \alpha$, using item (5) from the previous lemma.
\end{proof}    

\section{Initial algebras for endofunctors}

\begin{rema}{ccocomplete} In the remainder of the paper \ct{C} will be a cocomplete category with chosen colimits: so we assume that there is an operation assigning to each small diagram in \ct{C} a colimiting cocone.
\end{rema}

The goal of this section is to give a constructive proof of the following result using the theory of ordinals introduced in the previous section.

\begin{theo}{adamek}{\rm (Ad\'amek)}
    Let $F$ be an endofunctor on a cocomplete category \ct{C}, and $\lambda$ be a limit ordinal. If $F$ preserves colimits of shape $\Downarrow \lambda$, then $F$ has an initial algebra. 
\end{theo}  

To make sense of this result we need a good constructive notion of limit ordinal. The notion that makes our proof work is the following.

\begin{defi}{limitordinal} An ordinal $\lambda$ will be called a \emph{limit ordinal} if for any $\alpha, \beta \lt \lambda$ there is an ordinal $\gamma \lt \lambda$ such that $\alpha, \beta \lt \gamma$.
\end{defi}

We will give a detailed proof of this result, despite the fact that with a slightly different presentation, the proof can also be found in the work of Pitts and Steenkamp \cite{pittssteenkamp21}. The reason we include a detailed proof here is because the results in the subsequent sections build on this argument.  Like Pitts and Steenkamp, we will use the language of algebraic chains, with the caveat that constructively they are not chains, because they are not linearly ordered.  (See \cite{pittssteenkamp21,BourkeEquipping,bergetal25}; the method is due to Koubek and Reiterman \cite{Koubek1979Categorical}.) However, instead of their term ``inflationary iteration'' we prefer ``progressive iteration'', because it sounds more positive.

In the remainder of this section $F: \ct{C} \to \ct{C}$ is an endofunctor on a cocomplete category \ct{C}.

\begin{defi}{simplealgebraicchain} Let $\alpha$ be an ordinal. An \emph{algebraic chain} of length $\alpha$ is a pair $(X, x)$ consisting of:
\begin{enumerate}
    \item[(a)] a functor $X: \Downarrow \alpha \to \ct{C}$,
    \item[(b)] for any $\gamma \lt \beta \lt \alpha$ a map $x_{\gamma}^{\beta}: FX_\gamma \to X_\beta$ in \ct{C},
\end{enumerate}
such that for all $\delta \lt \gamma \lt \beta \lt \alpha$ the two diagrams on the left commute:
    \begin{displaymath}
        \begin{array}{ccc}
        \begin{tikzcd}
            FX_\gamma \ar[r, "x_{\gamma}^{\beta}"] & X_\beta \\
            FX_\delta \ar[u, "FX^\gamma_\delta"] \ar[ur, "x_{\delta}^{\beta}"']
        \end{tikzcd} &
        \begin{tikzcd}
            & X_\beta \\
            FX_\delta \ar[r, "x^\gamma_\delta"'] \ar[ur, "x_{\delta}^{\beta}"] & X_\gamma \ar[u, "X^\beta_\gamma"']
        \end{tikzcd} & 
        \begin{tikzcd}
            FX_\gamma \ar[d, "x^\beta_\gamma"] \ar[r, "F\varphi_\gamma"] & FY_\gamma \ar[d, "y^\beta_\gamma"] \\
            X_\beta \ar[r, "\varphi_\beta"] & Y_\beta
        \end{tikzcd} 
    \end{array}
    \end{displaymath}
    A morphism $\varphi: (X,x) \to (Y, y)$ of algebraic chains of length $\alpha$ is a natural transformation $\varphi: X \to Y$ making the diagram on the right commute for any $\gamma \lt \beta \lt \alpha$.
\end{defi}
  
\begin{lemm}{simplediagram}
    Let $(X,x)$ be an algebraic chain of length $\alpha$. If $\beta \leq \alpha$, then there is a diagram $\Downarrow \beta \to \ct{C}$ obtained by sending $\gamma$ to $FX_\gamma$ and $\delta \leq \gamma$ to $FX_\delta^\gamma$. If $\beta \lt \alpha$, then we have a cocone on this diagram given by $(X_\beta, x_\gamma^\beta: FX_\gamma \to X_\beta)$.
\end{lemm}

\begin{defi}{simpleprogiteration}
    An algebraic chain $(X, x)$ of length $\alpha$ will be called a \emph{progressive iteration of $F$ over $\alpha$} if for each $\beta \lt \alpha$ the cocone from the lemma above is the chosen colimit.
\end{defi}

\begin{lemm}{restrictionofsimpleprog}
  If $(X, x)$ is an algebraic chain of length $\alpha$ and $\beta \lt \alpha$, then so is \[ (X, x) \upharpoonright \beta = (X \upharpoonright \Downarrow \beta, x \upharpoonright \Downarrow \beta). \] Moreover, if $(X, x)$ is a progressive iteration of $F$, then so is $(X, x) \upharpoonright \beta$.
\end{lemm}

\begin{prop}{simpleprogitinitial}
    If an algebraic chain $(X, x)$ is a progressive iteration of $F$ over $\alpha$, then it is initial in the category of algebraic chains.
\end{prop} 
\begin{proof}
    Suppose $(X,x)$ and $(Y, y)$ are algebraic chains and $(X, x)$ is a progressive iteration of $F$ over $\alpha$. We will construct by recursion on $\beta \lt \alpha$ maps $\varphi_\beta: X_\beta \to Y_\beta$ making the following two diagrams commute
    \begin{equation} \label{construction}
        \begin{array}{ccc}
            \begin{tikzcd}
                X_\gamma \ar[d, "X_\gamma^\beta"] \ar[r, "\varphi_\gamma"] & Y_\gamma \ar[d, "Y_\gamma^\beta"] \\
                X_\beta \ar[r, "\varphi_\beta"] & Y_\beta 
            \end{tikzcd} & & 
            \begin{tikzcd}
                FX_\gamma \ar[d, "x_\gamma^\beta"] \ar[r, "F\varphi_\gamma"] & FY_\gamma \ar[d, "y_\gamma^\beta"] \\
                X_\beta \ar[r, "\varphi_\beta"] & Y_\beta 
            \end{tikzcd}
        \end{array}
    \end{equation}
    for any $\gamma \lt \beta$. Since $(X,x)$ is a progressive iteration, a unique map $\varphi_\beta$ making the diagram on the right commutes will exist as soon as $(Y_\beta, y_\gamma^\beta \circ F\varphi_\gamma)$ is a cocone on the diagram defined in the lemma above. This is the case, because for any $\delta \lt \gamma$ we have
        \begin{align*}
        y_\gamma^\beta \circ F\varphi_\gamma \circ FX_\delta^\gamma & = y_\gamma^\beta \circ FY^\gamma_\delta \circ F\varphi_\delta \quad & \textrm{(induction hypothesis)}\\
        & = y_\delta^\beta \circ F\varphi_\delta & \textrm{($(Y, y)$ algebraic chain)}
        \end{align*}

    It remains to show that the unique map $\varphi_\beta$ making the diagram on the right commute in (\ref{construction}) makes the diagram on the left commute as well. For this we use that $X_\gamma$ is a colimit and therefore to prove that the two composition along the outsides of the square on the left commute, it suffices to prove that they commute upon precomposition with any $x_\delta^\gamma$ for $\delta \lt \gamma$. To see that that indeed happens, we calculate:
    \begin{align*}
        \varphi_\beta \circ X^\beta_\gamma \circ x_\delta^\gamma & =  \varphi_\beta \circ x_\delta^\beta \quad & \textrm{($(X,x)$ algebraic chain)}\\
        & = y_\delta^\beta \circ F\varphi_\delta & \textrm{(righthand square in (\ref{construction}) commutes)} \\
        & = Y_\gamma^\beta \circ y_\delta^\gamma \circ F\varphi_\delta & \textrm{($(Y, y)$ algebraic chain)} \\
        & = Y^\beta_\gamma \circ \varphi_\gamma \circ x_\delta^\gamma & \textrm{(righthand square in (\ref{construction}) commutes)} 
    \end{align*}
    This finishes the proof.
\end{proof}

\begin{prop}{existenceofuniquesimpleprogit}
   For each ordinal $\alpha$ there exists a unique progressive iteration of $F$ over $\alpha$.
\end{prop}
\begin{proof}
    We prove the statement by induction on $\alpha$. In view of \reflemm{decompositionofsup} we can split up the argument in two steps:
    \begin{enumerate}
      \item[(1)] If a unique progressive iteration over $\alpha$ exists, then the same is true for $s\alpha$.
      \item[(2)] If $A$ is a downward closed set of ordinals and for each $\alpha \in A$ a unique progressive iteration over $\alpha$ exists, then the same is true for $\bigcurlyvee A$. 
    \end{enumerate}

    (1) Suppose $(X', x')$ is the unique progressive iteration over $\alpha$. Then we define the (unique) progressive iteration $(X, x)$ of length $s\alpha$ by putting taking $(X', x')$ and extending it as follows: we define $(X_\alpha, x_{\beta}^\alpha: FX_\beta \to X_\alpha)$ to be the chosen colimit of the diagram from \reflemm{simplediagram}. 

    We still need to construct a map $X_\beta^\alpha: X_\beta \to X_\alpha$ for any $\beta \lt \alpha$. However, since $X_\beta$ is a colimit on a diagram for which $(X_\alpha, x^\alpha_\gamma)$ is a cocone, there will be unique map $X_\beta^\alpha$ such that $X_\beta^\alpha \circ x_\gamma^\beta = x_\gamma^\alpha$ for all $\gamma \lt \beta$.

    It remains to show that $X$ is still a functor; that is, that $X_\beta^\alpha \circ X_\gamma^\beta = X_\gamma^\alpha$ holds for all $\gamma \lt \beta \lt \alpha$. For this it suffices to prove that these maps become equal upon precomposition with $x_\delta^\gamma$ for any $\delta \lt \gamma$, which is easy to see:
    \[ X_\beta^\alpha \circ X_\gamma^\beta \circ x_\delta^\gamma = X_\beta^\alpha \circ x_\delta^\beta = x_\delta^\alpha = X_\gamma^\alpha \circ x_\delta^\gamma. \]

    (2) Next suppose that for each $\alpha \in A$ we have a unique progressive iteration $(X^\alpha, x^\alpha)$. We would like to define a progressive iteration of length $\bigcurlyvee A$ as follows: we put $X_\beta := X^\alpha_\beta$ where $\alpha \in A$ is such that $\beta \lt \alpha$ and $X_\gamma^\beta := (X^\alpha)_\gamma^\beta$ and $x_\gamma^\beta := (x^\alpha)_\gamma^\beta$ where $\alpha \in A$ is such that $\gamma \lt \beta \lt \alpha$. It is clear that we have no choice but to define the progressive iteration of length $\bigcurlyvee A$ this way; however, what needs to be shown is that this definition does not depend on the choice of $\alpha$.
    
    So suppose $\alpha, \alpha' \in A$ are given and $\beta \lt \alpha, \alpha'$. It follows that $\beta \in A$ and therefore we obtain from the uniqueness of progressive iterations of length $\beta$ that $(X^\alpha)\upharpoonright \beta = X^\beta = X^{\alpha'} \upharpoonright \beta$. The previous argument then tells us that we must also have $X^\alpha_\beta = X^{\alpha'}_\beta$ as well as $(X^\alpha)^\beta_\gamma = (X^{\alpha'})^\beta_\gamma$ and $(x^\alpha)^\beta_\gamma = (x^{\alpha'})^\beta_\gamma$ for all $\gamma \lt \beta$, which is exactly what we needed to prove.
\end{proof}    

\begin{theo}{adamekagain} {\rm (Constructive Ad\'amek Theorem)}
    Let \ct{C} be a cocomplete category and $\lambda$ be a limit ordinal. If $F: \ct{C} \to \ct{C}$ is an endofunctor preserving colimits of shape $\lambda$, then the initial $F$-algebra exists.
\end{theo}
\begin{proof}
    Let $(X,x)$ be the unique progressive iteration of $F$ of length $\lambda$. We will construct the initial $F$-algebra by taking the colimit $I = {\rm colim} \, X$ and constructing a map $i: FI \to I$. Note that $I$ being a colimit means that it comes equipped with a colimiting cocone $(I, i_\alpha: X_\alpha \to I)$ on the diagram $X$.

    Our first step is to show that there is a unique map $i: FI \to I$ such that for any pair of ordinals $\beta \lt \alpha \lt \lambda$ we have a commutative square of the form
    \begin{equation} \label{descommsq}
        \begin{tikzcd}
            FX_\beta \ar[r, "Fi_\beta"] \ar[d, "x^\alpha_\beta"] & FI \ar[d, "i"]\\
            X_\alpha \ar[r, "i_\alpha"] & I.
        \end{tikzcd}
    \end{equation}
    For this we use that $I$ is the colimit of $X$ and this colimit is preserved by $F$: that means that in order to construct a map $FI \to I$ it suffices to construct for each $\beta \lt \lambda$ a map $h_\beta: FX_\beta \to I$ such that $h_\beta \circ FX_\gamma^\beta = h_\gamma$ holds whenever $\gamma \lt \beta \lt \alpha$. We choose $h_\beta := i_\alpha \circ x_\beta^\alpha$ where $\alpha$ is some ordinal such that $\beta \lt \alpha \lt \lambda$. The existence of such an ordinal $\alpha$ is guaranteed by our assumption that $\lambda$ is a limit ordinal. But it also guarantees this definition does not depend on the choice of $\alpha$. For suppose $\alpha, \alpha' \lt \lambda$ are such that $\beta \lt \alpha, \alpha'$. Then there is some $\alpha'' \lt \lambda$ such that $\alpha, \alpha' \lt \alpha''$ from which we obtain a commutative diagram of the following shape:
    \begin{displaymath}
        \begin{tikzcd}
            & X_{\alpha} \ar[dr, "i_{\alpha}", bend left] \ar[d, "X^{\alpha''}_\alpha"'] \\
        FX_\beta \ar[ur, "x^\alpha_\beta", bend left] \ar[r, "x^{\alpha''}_\beta"'] \ar[dr, "x^{\alpha'}_\beta"', bend right] & X_{\alpha''} \ar[r, "i_{\alpha''}"] & I \\
        & X_{\alpha'} \ar[ur, bend right, "i_{\alpha'}"'] \ar[u, "X^{\alpha'}_\alpha"] 
        \end{tikzcd}
    \end{displaymath}
    Therefore the definition of $h_\beta$ does not depend on the choice of $\alpha$. Moreover, we have
    \[ h_\beta \circ FX_\gamma^\beta = i_\alpha \circ x^\alpha_\beta \circ FX_\gamma^\beta = i_\alpha \circ x_\gamma^\alpha = h_\gamma, \]
    and therefore there exists a unique map $i: FI \to I$ making the diagram (\ref{descommsq}) commute.

    It remains to show that $(I, i)$ is the initial $F$-algebra. So let $(Y, y: FY \to Y)$ be an $F$-algebra. Note that we can also regard $(Y, y)$ as an algebraic chain of length $\lambda$ with $Y_\alpha := Y$, $Y_\beta^\alpha = 1_Y$ and $y_\beta^\alpha = y$; we will denote this chain with $(Y, y)$ as well. By \refprop{simpleprogitinitial} there exists a unique morphisms of chains $(X,x) \to (Y, y)$; in other words, there exists a unique collection of morphisms $(\varphi_\alpha: X_\alpha \to Y \, : \, \alpha \lt \lambda)$ making the two squares below commute:
    \begin{equation} \label{twosqmorphism}
        \begin{array}{ccc}
            \begin{tikzcd}
                X_\beta \ar[d, "X_\beta^\alpha"] \ar[r, "\varphi_\beta"] & Y \ar[d, "1"] \\
                X_\alpha \ar[r, "\varphi_\alpha"] & Y
            \end{tikzcd} & & 
            \begin{tikzcd}
                FX_\beta \ar[d, "x_\alpha^\beta"] \ar[r, "F\varphi_\beta"] & FY \ar[d, "y"] \\
                X_\alpha \ar[r, "\varphi_\alpha"] & Y
            \end{tikzcd}
        \end{array}
    \end{equation}
    The commutativity of the square on the left implies that there exists a unique map $\varphi: I \to Y$ such that $\varphi_\alpha = \varphi \circ i_\alpha$ for all $\alpha \lt \lambda$. To show that this is a morphism of $F$-algebras it suffices to prove that $\varphi \circ i$ and $y \circ F\varphi$ become equal upon precomposition with $Fi_\beta$ for any $\beta \lt \lambda$. However, we have
    \begin{align*}
        \varphi \circ i \circ Fi_\beta & =  \varphi \circ i_\alpha \circ x_\beta^\alpha \quad & \textrm{(commutativity of (\ref{descommsq}))}\\
        & =  \varphi_\alpha \circ x_\beta^\alpha & \textrm{(definition $\varphi$)}\\
        & =  y \circ F\varphi_\beta & \textrm{(righthand square in (\ref{twosqmorphism})}\\
        & =  y \circ F\varphi \circ Fi_\beta & \textrm{(definition $\varphi$)}
    \end{align*}
    where $\alpha$ is any ordinal such that $\beta \lt \alpha \lt \lambda$.

    Finally, if $\psi: I \to Y$ is any morphism of $F$-algebras, we can define $\psi_\alpha := \psi \circ i_\alpha$ and we can prove $\psi_\alpha = \varphi_\alpha$ by showing that the $\psi_\alpha$ also make the two diagrams in (\ref{twosqmorphism}) commute. To see this we note that
    \[ \psi_\alpha \circ X_\beta^\alpha = \psi \circ i_\alpha \circ X_\beta^\alpha = \psi \circ i_\beta = \psi_\beta, \]
    while
    \[ \psi_\alpha \circ x_\alpha^\beta = \psi \circ i_\alpha \circ x_\alpha^\beta = \psi \circ i \circ Fi_\beta = y \circ F\psi \circ Fi_\beta = y \circ F\psi_\beta. \]
    Since the $i_\alpha$ are part of a colimiting cone and hence jointly monic, this finishes the proof.
\end{proof}

\begin{coro}{monadicityFalg}
    Let \ct{C} be a cocomplete category and $\lambda$ be a limit ordinal. If $F: \ct{C} \to \ct{C}$ is an endofunctor preserving colimits of shape $\Downarrow \lambda$, then the forgetful functor
    \[ U: F\text{\rm -Alg} \to \ct{C} \]
    is monadic.
\end{coro}
\begin{proof}
    To show that $U$ has a left adjoint, we need to verify that for each object $A$ in \ct{C}, the category $A \downarrow U$ has an initial object. But since this category is isomorphic to the category of algebras for the endofunctor $X \mapsto A + FX$, and this endofunctor also preserves colimits of shape $\Downarrow \lambda$, this follows from the previous result.

    In order to apply Beck's monadicity theorem \cite[Theorem 4.4.4]{borceux94ii}, it remains to check that $U$ is conservative and $U$ creates coequalizers of $U$-split pairs. We have relegated a proof of this to an appendix: see \reftheo{checkformonadicity}.
\end{proof}

\section{Initial algebras for pointed endofunctors}

In this section we will give a constructive proof of Ad\'amek's result for pointed endofunctors. That is, we will assume that the endofunctor $F: \ct{C} \to \ct{C}$ comes equipped with a natural transformation (``pointing'') $\eta: 1_\ct{C} \to F$. In this case we are interested in constructing the initial algebra for the pointed endofunctor $(F, \eta)$, which is the initial object in the full subcategory of $F$-algebras $(X, x: FX \to X)$ satisfying $x \circ \eta_X = 1_X$. We will show that these exist when $F$ preserves suitable colimits by adapting the proof from the previous section. The main challenge is to figure out how the notions of algebraic chain and progressive iteration need to be modified in this pointed setting, so that the same proof idea goes through.

In this section we fix a pointed endofunctor $(F, \eta)$ on a cocomplete category \ct{C}.

\begin{defi}{pointedchain} An \emph{algebraic chain} $(X, x)$ of length $\alpha$ will be called \emph{pointed} if not only the two diagrams on the left commute for any $\delta \lt \gamma \lt \beta \lt \alpha$ (as in any algebraic chain), but also the third one on the right commutes for any $\gamma \lt \beta \lt \alpha$:
    \begin{equation} \label{pointedchaindef}
        \begin{array}{ccc}
        \begin{tikzcd}
            FX_\gamma \ar[r, "x_{\gamma}^{\beta}"] & X_\beta \\
            FX_\delta \ar[u, "FX^\gamma_\delta"] \ar[ur, "x_{\delta}^{\beta}"']
        \end{tikzcd} &
        \begin{tikzcd}
            & X_\beta \\
            FX_\delta \ar[r, "x^\gamma_\delta"'] \ar[ur, "x_{\delta}^{\beta}"] & X_\gamma \ar[u, "X^\beta_\gamma"']
        \end{tikzcd} &
        \begin{tikzcd}
            & X_\beta \\
            X_\gamma \ar[r, "\eta_{X_\gamma}"'] \ar[ur, "X_{\gamma}^{\beta}"] & FX_\gamma \ar[u, "x^\beta_\gamma"']
        \end{tikzcd}
    \end{array}
\end{equation}
The category of pointed algebraic chains is the full subcategory of the category of algebraic chains on the chains that are pointed. In other words, we do not add conditions on the morphisms of algebraic chains that take into account the existence of the pointing.
\end{defi}

\begin{rema}{functorialityfollows} For pointed algebraic chains the functoriality of the chains and the naturality of the morphisms becomes redundant. Indeed, the functoriality of $X$ follows from commutative triangles above:
    \begin{align*}
        X^\beta_\gamma \circ X^\gamma_\delta & =  X^\beta_\gamma \circ x^\gamma_\delta \circ \eta_{X_\delta} & \textrm{(third triangle)}\\
        & = x^\beta_\delta \circ \eta_{X_\delta} & \textrm{(second triangle)} \\
        & = X^\beta_\delta & \textrm{(third triangle)}
    \end{align*}

    Secondly, for a morphism of chains the following two diagrams need to commute
\begin{displaymath}
    \begin{array}{cc}
    \begin{tikzcd}
        X_\gamma \ar[d, "X^\beta_\gamma"] \ar[r, "\varphi_\gamma"] & Y_\gamma \ar[d, "Y^\beta_\gamma"] \\
        X_\beta \ar[r, "\varphi_\beta"] & Y_\beta
    \end{tikzcd} &
    \begin{tikzcd}
        FX_\gamma \ar[d, "x^\beta_\gamma"] \ar[r, "F\varphi_\gamma"] & FY_\gamma \ar[d, "y^\beta_\gamma"] \\
        X_\beta \ar[r, "\varphi_\beta"] & Y_\beta
    \end{tikzcd} 
\end{array}
\end{displaymath}
whenever $\gamma \lt \beta \lt \alpha$.
However, for pointed algebraic chains the commutativity of the naturality square on the left follows from the commutativity of the diagram on the right: indeed, it follows from the new requirement for pointed algebraic chains that we obtain the diagram on the left from the diagram on the right by sticking the naturality square
\begin{displaymath}
    \begin{tikzcd}
        X_\gamma \ar[d, "\eta_{X_\gamma}"] \ar[r, "\varphi_\gamma"] & FY_\gamma \ar[d, "\eta_{Y_\gamma}"] \\
        FX_\gamma \ar[r, "F\varphi_\gamma"] & FY_\gamma
    \end{tikzcd} 
\end{displaymath}
on top of it.
\end{rema}    

\begin{defi}{categoryD} Let $D(\beta)$ be the category which has as objects ordinals $\gamma \lt \beta$ and which has two morphisms $0,1:\delta \to \gamma$ if $\delta \lt \gamma$ and an identity morphism $1: \gamma \to \gamma$ (and no other morphisms). Composition is computed by taking the minimum.
\end{defi} 

\begin{lemm}{coconesfromchains}
Let $(X,x)$ be a pointed algebraic chain of length $\alpha$. If $\beta \leq \alpha$, then there is a diagram $D(\beta) \to \ct{C}$ obtained by sending $\gamma$ to $FX_\gamma$ and $0: \delta \to \gamma$ to $\eta_{X_\gamma} \circ x_\delta^\gamma$ and $1: \delta \to \gamma$ to $FX_\delta^\gamma$. If $\beta \lt \alpha$, then we have a cocone on this diagram given by $(X_\beta, x_\gamma^\beta: FX_\gamma \to X_\beta)$.
\end{lemm}
\begin{proof}
    For any $\varepsilon \lt \delta \lt \gamma$, we have the following commutative diagrams:
    \begin{displaymath}
        \begin{array}{cc}
        \begin{tikzcd}
            FX_\varepsilon \ar[r] \ar[dr, "x^\delta_\varepsilon"] \ar[drr, bend right = 50, "x^\gamma_\varepsilon"'] & FX_\delta \ar[r, "FX^\gamma_\delta"] &  FX_\gamma \\
            & X_\delta \ar[r, "X^\gamma_\delta"] \ar[u, "\eta_{X_\delta}"'] & X_\gamma \ar[u, "\eta_{X_\gamma}"']
        \end{tikzcd} &
        \begin{tikzcd}
            FX_\varepsilon \ar[r, "FX^\delta_\varepsilon"] \ar[dr, "x^\gamma_\varepsilon"'] & FX_\delta \ar[r] \ar[d, "x^\gamma_\delta"] & FX_\gamma \\
            & X_\gamma \ar[ur, "\eta_{X_\gamma}"']
        \end{tikzcd} \\
        \begin{tikzcd}
            FX_\varepsilon \ar[r] \ar[dr, "x^\delta_\varepsilon"] \ar[drr, bend right = 50, "x^\gamma_\varepsilon"']& FX_\delta \ar[r] \ar[dr, "x^\gamma_\delta"] & FX_\gamma \\
            & X_\delta \ar[r, "X^\gamma_\delta"] \ar[u, "\eta_{X_\delta}"'] & X_\gamma \ar[u, "\eta_{X_\gamma}"']
        \end{tikzcd}  
    \end{array}
    \end{displaymath}
    This shows that we have a diagram $D(\beta) \to \ct{C}$. 

    In addition, we have
    \[ x^\beta_\gamma \circ \eta_{X_\gamma} \circ x^\gamma_\delta = X^\beta_\gamma \circ x^\gamma_\delta = x^\beta_\delta \]
    and
    \[ x^\beta_\gamma \circ FX^\gamma_\delta = x^\beta_\delta, \]
    which shows $(X_\beta, x_\gamma^\beta: FX_\gamma \to X_\beta)$ is a cocone on this diagram.
\end{proof}

\begin{defi}{canonicalchain}
A pointed algebraic chain $(X, x)$ will be called a \emph{pointed progressive iteration} if for each $\beta \lt \alpha$ the cocone from \reflemm{coconesfromchains} is the chosen colimit.
\end{defi}

\begin{lemm}{restrictionofpointedchains}
    If $(X,x)$ is a pointed algebraic chain of length $\alpha$ and $\beta \leq \alpha$, then $(X,x) \upharpoonright \beta = (X \upharpoonright \Downarrow \beta, x \upharpoonright \Downarrow \beta)$ is a pointed algebraic chain of length $\beta$. This chain will be a pointed progressive iteration whenever $(X,x)$ is.
   \end{lemm}

\begin{prop}{fromcolimitstoinitial}
    If $(X,x)$ is a pointed progressive iteration, then $(X, x)$ is initial in the category of pointed algebraic chains.
\end{prop}
\begin{proof}
    Suppose $(X,x)$ is a pointed progressive iteration, and $(Y, y)$ is any pointed algebraic chain. As before, we construct maps $\varphi_\beta: X_\beta \to Y_\beta$ by induction on $\beta$, using that $X_\beta$ is a colimit and that we have already constructed maps $\varphi_\gamma: X_\gamma \to Y_\gamma$ for all $\gamma \lt \beta$ making 
    \begin{equation} \label{pointedIH}
        \begin{tikzcd}
            FX_\delta \ar[d, "x^\gamma_\delta"] \ar[r, "F\varphi_\delta"] & FY_\delta \ar[d, "y^\gamma_\delta"] \\
            X_\gamma \ar[r, "\varphi_\gamma"] & Y_\gamma
        \end{tikzcd} 
    \end{equation}
    commute whenever $\delta \lt \gamma$. To construct the unique map $\varphi_\beta$ making
    \begin{displaymath}
        \begin{tikzcd}
            FX_\gamma \ar[r, "F\varphi_\gamma"] \ar[d, "x_{\gamma}^{\beta}"] & FY_\gamma \ar[d, "y_{\gamma}^{\beta}"] \\
            X_\beta \ar[r, dotted, "\varphi_\beta"] & Y_\beta
        \end{tikzcd}
    \end{displaymath}
    commute it suffices to prove that $(Y_\beta, y^\beta_\gamma \circ F\varphi_\gamma$) is a cocone on the diagram $D(\beta) \to \ct{C}$ that we constructed above. In other words, we have to prove that \[ y^\beta_\gamma \circ F\varphi_\gamma \circ FX^\gamma_\delta = y_\delta^\beta \circ F\varphi_\delta \quad \textrm{and} \quad y^\beta_\gamma \circ F\varphi_\gamma \circ \eta_{X_\gamma} \circ x^\gamma_\delta = y_\delta^\beta \circ F\varphi_\delta. \]
    To see this, note that we have that
    \begin{align*}
        y^\beta_\gamma \circ F\varphi_\gamma \circ FX^\gamma_\delta & = y_\delta^\beta \circ FY_\delta^\gamma \circ F\varphi_\delta & \textrm{($\varphi$ natural by induction hypothesis)} \\
        & = y_\delta^\beta \circ F\varphi_\delta & \textrm{($(Y, y)$ algebraic chain)}
    \end{align*}
    as well as
    \begin{align*}
        y^\beta_\gamma \circ F\varphi_\gamma \circ \eta_{X_\gamma} \circ x^\gamma_\delta & = y_\gamma^\beta \circ \eta_{Y_\gamma} \circ \varphi_\gamma \circ x_\delta^\gamma & \textrm{(naturality of $\eta$)} \\
        & = y_\gamma^\beta \circ \eta_{Y_\gamma} \circ y_\delta^\gamma \circ F\varphi_\delta & \textrm{(commutativity of (\ref{pointedIH}))} \\
        & = y_\delta^\beta \circ F\varphi_\delta & \textrm{($(Y, y)$ pointed algebraic chain)}
    \end{align*}
    which completes the proof.
\end{proof}

\begin{prop}{existencepointedprogiteration}
    For each $\alpha$ there exists a unique pointed progressive iteration of length $\alpha$.
\end{prop}
\begin{proof}
    We modify the proof of \refprop{existenceofuniquesimpleprogit}. We still prove the claim by induction on $\alpha$, but we will only discuss the successor step, as the proof of the other step can be copied verbatim. So suppose that $\alpha = s\beta$ and we have a unique pointed progressive iteration $(X, x)$ of length $\beta$. Then we have a diagram on $D(\beta)$, as in \reflemm{coconesfromchains}, and we have to define $(X_\beta, x^\beta_\gamma)$ as the chosen colimit of this diagram if we wish to extend the pointed progressive iteration to one of length $\alpha$. In addition, we are forced to define $X^\beta_\gamma := x^\beta_\gamma \circ \eta_{X_\gamma}$, which, together with the fact that $(X_\beta, x_\gamma^\beta)$ is a cocone on $D(\beta)$, gives us the commutativity of the three triangles in (\ref{pointedchaindef}).
\end{proof}

\begin{theo}{convergence} Let $(F, \eta)$ be a pointed endofunctor on a cocomplete category \ct{C}. If $\lambda$ is a limit ordinal and $F$ preserves colimits of shape $\Downarrow \lambda$, then the initial pointed $F$-algebra exists.
\end{theo}
\begin{proof}
    We have to make one small addition to the proof of \reftheo{adamekagain}. So let $(X, x)$ be the unique pointed progressive iteration of length $\lambda$ and $(I, i_\gamma)$ be the colimit of $X$. As before, we can construct a unique map $i: FI \to I$ such that for any $\gamma \lt \beta \lt \alpha$ the diagram
    \begin{displaymath}
        \begin{tikzcd}
            FX_\gamma \ar[r, "Fi_\gamma"] \ar[d, "x^\beta_\gamma"] & FI \ar[d, "i"] \\
            X_\beta \ar[r, "i_\beta"] & I
        \end{tikzcd}
    \end{displaymath}
    commutes. Once we show that $i$ is a pointed algebra, we can prove that $(I, i)$ is initial pointed algebra, essentially as in \reftheo{adamekagain}.
    
    In order to verify $i \circ \eta_I = 1_Y$, it suffices to check that $i \circ \eta_I \circ i_\gamma = i_\gamma$ for any $\gamma \lt \lambda$. But we have
    \begin{align*}
        i \circ \eta_I \circ i_\gamma & =  i \circ Fi_\gamma \circ \eta_{X_\gamma}  & \textrm{($\eta$ natural)}\\
        & =  i_\beta \circ x^\beta_\gamma \circ \eta_{X_\gamma} & \textrm{(diagram above)}\\
        & =  i_\beta \circ X^\beta_\gamma & \textrm{($(X, x)$ pointed)}\\
        & =  i_\gamma & \textrm{($(I, i)$ cocone)},
    \end{align*}
    where $\beta$ is any ordinal such that $\gamma \lt \beta \lt \lambda$. This completes the proof.
\end{proof}

\begin{coro}{monadicitypointedFalg}
    Let $(F, \eta)$ be a pointed endofunctor on a cocomplete category \ct{C}. If $\lambda$ is a limit ordinal and $F: \ct{C} \to \ct{C}$ is an endofunctor preserving colimits of shape $\Downarrow \lambda$, then the forgetful functor
    \[ U: (F, \eta)\text{\rm -Alg} \to \ct{C} \]
    is monadic.
\end{coro}
\begin{proof}
    To show that $U$ has a left adjoint, we need to verify that for each object $A$ in \ct{C}, the category $A \downarrow U$ has an initial object. This category is equivalent to the category of algebras for the endofunctor  endofunctor $X \mapsto A + FX$ with pointing $\inr \circ \eta_X$. Since this functor still preserves colimits of shape $\Downarrow \lambda$, the existence of the initial object follows from the previous result. The remaining verifications needed for applying Beck's monadicity theorem can be found in the proof of \reftheo{checkformonadicity}.
\end{proof}

\section{Initial algebras for special endofunctors}

In this section we further modify the argument from the previous section to prove the existence of initial special algebras. The notion of a special algebra was introduced in \cite{bergetal25} in order to give a self-contained conceptual proof of the double-categorical small object argument. Our only reason for studying special algebras is that we will use them for our constructive small object argument in the next section; we are not aware of any other applications.

In this section we assume that we are given endofunctors $F$ and $G$ on a cocomplete category \ct{C} together with natural transformations $\eta: 1_\ct{C} \Rightarrow F$, $\sigma: G \Rightarrow FF$ and $\tau: G \Rightarrow F$.

\begin{defi}{specialalgebra}
    An algebra $(A, a)$ for the pointed endofunctor $(F,\eta)$ will be called \emph{special} if the following diagram commutes:
    \begin{displaymath}
        \begin{tikzcd}
            G A \ar[r, "\sigma_A"] \ar[d, "\tau_A"] & FFA \ar[r, "Fa"] & FA \ar[d, "a"] \\
            FA \ar[rr, "a"] & & A.
        \end{tikzcd}
    \end{displaymath}
    The category of special algebraic chains is a full subcategory of the category of algebraic chains on the endofunctor $F$.
\end{defi}    

We will construct the initial special algebra again as a colimit of a suitable chain.

\begin{defi}{specialchain}
    An algebraic chain $(X,x)$ of length $\alpha$ for the pointed endofunctor $(F,\eta)$ is \emph{special} if for any $\delta \lt \gamma \lt \beta \lt \alpha$, the diagram
        \begin{displaymath}
            \begin{tikzcd}
                G X_\delta \ar[r, "\sigma_{X_\delta}"] \ar[d, "\tau_{X_\delta}"] & FFX_\delta \ar[r, "Fx^\gamma_\delta"] & FX_\gamma\ar[d, "x_\gamma^\beta"] \\
            FX_\delta \ar[rr, "x_\delta^\beta"] & & X_\beta
            \end{tikzcd}
        \end{displaymath}
        commutes.
\end{defi}

Since for a pointed algebraic chain $x_\delta^\beta = x^\beta_\gamma \circ \eta_{X_\gamma} \circ x_\delta^\gamma$, the additional condition for a special algebraic chain is equivalent to saying that 
\begin{equation} \label{fork}
    \begin{tikzcd}
        G X_\delta \ar[shift left, rr, "Fx^\gamma_\delta \circ \sigma_{X_\delta}"] \ar[rr, shift right, "\eta_{X_\gamma} \circ x^\gamma_\delta \circ \tau_{X_\delta}"'] & & FX_\gamma \ar[r, "x^\beta_\gamma"] & X_\beta
    \end{tikzcd}
\end{equation}
is a fork. If $(X, x)$ is a special chain, we will write $x_{\delta,\gamma}^\beta$ for both composites in this fork. Note that if $\varphi: (X, x) \to (Y, y)$ is a morphism of special chains, then
\[ \varphi_\beta \circ x_{\delta, \gamma}^\beta = y_{\delta, \gamma}^\beta \circ G(\varphi_\delta). \]

\begin{defi}{categoryS} Let $S(\beta)$ be the category which has as objects ordinals $\gamma \lt \beta$ and pairs of ordinals $(\delta, \gamma)$ with $\delta \lt \gamma \lt \beta$. Besides identity arrows, this category has two maps $0, 1: \psi \to \gamma$ where either (i) $\psi = \delta$ and $\delta \lt \gamma$ or (ii) $\psi = (\delta, \delta')$ and $\delta \lt \delta' \leq \gamma$, and no other morphisms. Composition is computed by taking the minimum.
\end{defi} 

\begin{lemm}{coconesfromspecialchains}
Let $(X,x)$ be a special algebraic chain of length $\alpha$. If $\beta \leq \alpha$, then there is a diagram $S(\beta) \to \ct{C}$ obtained by sending $\gamma$ to $FX_\gamma$ and $(\delta, \delta')$ to $G X_\delta$; in addition, we send
\begin{itemize}
    \item $0: \delta \to \gamma$ to $\eta_{X_\gamma} \circ x_\delta^\gamma$;
    \item $1: \delta \to \gamma$ to $FX_\delta^\gamma$;  
    \item $0: (\delta, \delta') \to \gamma$ to $\eta_{X_\gamma} \circ x^\gamma_\delta \circ \tau_{X_\delta}$;
    \item $1: (\delta, \delta') \to \gamma$ to $Fx_\delta^\gamma \circ \sigma_{X_\delta}$;
\end{itemize} 
If $\beta \lt \alpha$, then we have a cocone on this diagram with vertex $X_\beta$ and maps $x^\beta_\gamma: FX_\gamma \to X_\beta$ and $x_{\delta, \delta'}^\beta: GX_\delta \to X_\beta$.
\end{lemm}

\begin{defi}{specialcanonicalchain}
A special algebraic chain $(X, x)$ of length $\alpha$ will be called a \emph{special  progressive iteration} if for each $\beta \lt \alpha$ the cocone from \reflemm{coconesfromspecialchains} is the chosen colimit.
\end{defi}

\begin{prop}{specialfromcolimitstoinitial}
    If $(X,x)$ is a special progressive iteration of length $\alpha$, then $(X, x)$ is initial in the category of special algebraic chains.
\end{prop}
\begin{proof}
    Suppose $(X,x)$ is a special progressive iteration, and $(Y, y)$ is any special chain. 
    
    We construct maps $\varphi_\beta: X_\beta \to Y_\beta$ by induction on $\beta$, using that $X_\beta$ is a colimit and that we have already constructed maps $\varphi_\gamma: X_\gamma \to Y_\gamma$ for all $\gamma \lt \beta$ making 
    \begin{equation} \label{earlier}
        \begin{tikzcd}
            FX_\delta \ar[d, "x^\gamma_\delta"] \ar[r, "F\varphi_\delta"] & FY_\delta \ar[d, "y^\gamma_\delta"] \\
            X_\gamma \ar[r, "\varphi_\gamma"] & Y_\gamma
        \end{tikzcd} 
    \end{equation}
    commute whenever $\delta \lt \gamma$. To construct the unique map $\varphi_\beta$ making
    \begin{displaymath}
        \begin{tikzcd}
            FX_\gamma \ar[r, "F\varphi_\gamma"] \ar[d, "x_{\gamma}^{\beta}"] & FY_\gamma \ar[d, "y_{\gamma}^{\beta}"] \\
            X_\beta \ar[r, dotted, "\varphi_\beta"] & Y_\beta
        \end{tikzcd}
    \end{displaymath}
    commute for all $\gamma \lt \beta$ it suffices to prove that \[ (Y_\beta, (y^\beta_\gamma \circ F\varphi_\gamma)_\gamma, (y^\beta_{\delta, \delta'} \circ G\varphi_\delta)_{\delta, \delta'}) \] is a cocone on the diagram $S(\beta) \to \ct{C}$ that we constructed above. This means that it remains to verify four equalities:
    \begin{eqnarray*} 
        y^\beta_\gamma \circ F\varphi_\gamma \circ FX^\gamma_\delta & = & y_\delta^\gamma \circ F\varphi_\delta \\
        y^\beta_\gamma \circ F\varphi_\gamma \circ \eta_{X_\gamma} \circ x^\gamma_\delta & = & y_\delta^\gamma \circ F\varphi_\delta \\
        y^\beta_\gamma \circ F\varphi_\gamma \circ \eta_{X_\gamma} \circ x^\gamma_\delta \circ \tau_{X_\delta} & = & y_{\delta, \delta'}^\beta \circ G\varphi_\delta \\
        y^\beta_\gamma \circ F\varphi_\gamma \circ Fx_\delta^\gamma \circ \sigma_{X_\delta} & = & y_{\delta, \delta'}^\beta \circ G\varphi_\delta
    \end{eqnarray*}
    where $\delta' \leq \gamma$. The first two can be shown as in \refprop{fromcolimitstoinitial}, so we will prove the latter two. For the first we calculate
    \begin{align*}
        y_{\delta, \delta'}^\beta \circ G \varphi_\delta & = y^\beta_{\delta'} \circ \eta_{Y_{\delta'}} \circ y_\delta^{\delta'} \circ \tau_{Y_\delta} \circ G \varphi_\delta \\
        & = y_\gamma^\beta \circ FY_{\delta'}^\gamma \circ \eta_{Y_{\delta'}} \circ y_\delta^{\delta'} \circ F\varphi_\delta \circ \tau_{X_\delta} & \mbox{($(Y, y)$ chain and $\tau$ natural)}\\
        & = y_\gamma^\beta \circ \eta_{Y_\gamma} \circ Y^\gamma_{\delta'} \circ \varphi_{\delta'} \circ x_\delta^{\delta'} \circ \tau_{X_\delta} & \mbox{($\eta$ natural and (\ref{earlier}))} \\
        & = y_\gamma^\beta \circ \eta_{Y_\gamma} \circ \varphi_\gamma \circ X^\gamma_{\delta'} \circ x_\delta^{\delta'} \circ \tau_{X_\delta} & (\ref{earlier})\\
        & = y^\beta_\gamma \circ F\varphi_\gamma \circ \eta_{X_\gamma} \circ x^\gamma_\delta \circ \tau_{X_\delta} & \mbox{($\eta$ natural and $(X, x)$ chain)}
    \end{align*}
    while the second is shown as follows:
    \begin{align*}    
        y_{\delta, \delta'}^\beta \circ G\varphi_\delta & = y_{\delta'}^\beta \circ Fy_\delta^{\delta'} \circ \sigma_{Y_\delta} \circ G\varphi_\delta \\
        & = y_\gamma^\beta \circ FY_{\delta'}^\gamma \circ Fy_\delta^{\delta'} \circ FF\varphi_\delta \circ \sigma_{X_\delta} & \mbox{($(Y, y)$ chain and $\sigma$ natural)} \\
        & = y_\gamma^\beta \circ Fy_\delta^{\gamma} \circ FF\varphi_\delta \circ \sigma_{X_\delta} & \mbox{($(Y, y)$ chain)} \\
        & = y^\beta_\gamma \circ F\varphi_\gamma \circ Fx_\delta^\gamma \circ \sigma_{X_\delta} & (\ref{earlier})
    \end{align*}
    This completes the proof.
\end{proof}

The proof of the following proposition should now be routine.

\begin{prop}{specialexistence}
    For each $\alpha$ there exists a unique special progressive iteration of length $\alpha$.
\end{prop}

\begin{theo}{specialconvergence}
    Suppose $\lambda$ is a limit ordinal and both $F$ and $G$ preserve colimits of shape $\Downarrow \lambda$. If $(X,x)$ is the unique special progressive iteration of length $\lambda$, then $I = {\rm colim} \, X$ can be equipped with structure of a pointed $F$-algebra which will make it initial in the category of special $F$-algebras.
\end{theo}
\begin{proof}
    Since $I = {\rm colim} \, X$ we have maps $(i_\beta: X_\beta \to I)_{\beta \lt \alpha}$ such that $i_\gamma = i_\beta \circ X^\beta_\gamma$ for any $\gamma \lt \beta \lt \alpha$. As before, we can construct a unique map $i: FI \to I$ making the diagram
    \begin{equation} \label{luckyus}
        \begin{tikzcd}
            FX_\gamma \ar[r, "Fi_\gamma"] \ar[d, "x^\beta_\gamma"] & FI \ar[d, "i"] \\
            X_\beta \ar[r, "i_\beta"] & I
        \end{tikzcd}
    \end{equation}
    commute for any $\gamma \lt \beta \lt \alpha$, and as we have seen in the proof of \reftheo{convergence} this turns $I$ into an algebra for the pointed endofunctor $(F,\eta)$.

    Our next step is to show that the algebra $(I, i)$ is special, that is, that
    \[ i \circ Fi \circ \sigma_I = i \circ \tau_I \]
    holds. Using that $G$ preserves colimits, it suffices to show that these two maps become equal upon precomposition with $G i_\delta$ with $\delta \lt \alpha$. To see that that is the case, we use that $(X,x)$ is a special algebraic chain and we calculate:
    \begin{align*}
        i \circ Fi \circ \sigma_I \circ Gi_\delta & = i \circ Fi \circ FFi_\delta \circ \sigma_{X_\delta} & \mbox{($\sigma$ natural)}\\
        & = i\circ Fi_\gamma \circ Fx_\delta^\gamma \circ \sigma_{X_\delta} & \mbox{(\ref{luckyus})} \\
        & = i_\beta \circ x^\beta_\gamma \circ Fx_\delta^\gamma \circ \sigma_{X_\delta} & \mbox{(\ref{luckyus})} \\
        & = i_\beta \circ x^\beta_\delta \circ \tau_{X_\delta} & \mbox{($(X, x)$ special)}\\
        & = i \circ Fi_\delta \circ \tau_{X_\delta} & \mbox{(\ref{luckyus})} \\
        & = i \circ \tau_{I} \circ Gi_\delta & \mbox{($\tau$ natural)}
    \end{align*}
    In this proof we have also used that $\alpha$ is a limit ordinal to deduce the existence of $\gamma$ and $\beta$ such that $\delta \lt \gamma \lt \beta \lt \alpha$.

    It remains to show that $(I, i: F I \to I)$ is initial among the special $F$-algebras; however, this can be shown by an argument almost identical to the one given in \reftheo{adamekagain}.
\end{proof}

\begin{coro}{monadicityspecialFalg}
    Suppose $\lambda$ is a limit ordinal and both $F$ and $G$ preserve colimits of shape $\Downarrow \lambda$. The forgetful functor
    \[ U: \text{\rm SpAlg}(F,G) \to \ct{C} \]
    assigning to each special algebra its underlying object is monadic.
\end{coro}
\begin{proof}
    To show that $U$ has a left adjoint, we need to verify that for each object $A$ in \ct{C}, the category $A \downarrow U$ has an initial object. We have already seen in \refcoro{monadicitypointedFalg} that the category of pointed $F$-algebra is equivalent to the category of algebras for the endofunctor $F^A$ defined by $X \mapsto A + FX$ with pointing $\eta_X^A = \inr \circ \eta_X$. In the present situation we also have natural transformations
    \begin{align*}
        \sigma^A: G \Rightarrow F^AF^A & \mbox{ where } \sigma_X^A := \inr \circ F\inr \circ \sigma_X \\
        \tau^A: G \Rightarrow F^A & \mbox{ where } \tau_X^A := \inr \circ \tau_X
    \end{align*}
    and a special algebra for $(F^A, G, \eta_A, \sigma_A, \tau_A)$ is nothing but a special algebra $X$ equipped with a morphism $\alpha: A \to X$. Since $F^A$ and $G$ still preserves colimits of shape $\Downarrow \lambda$, the existence of the initial object in $U \downarrow A$ follows from the previous result. The remaining verifications needed for applying Beck's monadicity theorem can be found in the proof of \reftheo{checkformonadicity}.
\end{proof}

\section{A constructive small object argument}

After the work done in the previous section, we can give a quick constructive proof of a double-categorical small object argument. We assume familiarity with \cite{bergetal25}. The main idea is to introduce a suitable notion of $\lambda$-presentability.

\begin{defi}{presentability}
Let $\ct{C}$ be a category and $\lambda$ be a limit ordinal. We say that an object $A$ in $\ct{C}$ is \emph{$\lambda$-small} if
    \[ {\rm Hom}(A, -): \ct{C} \to {\bf Sets} \]
    preserves colimits of shape $\Downarrow \lambda$.
\end{defi}

\begin{theo}{smallobjargstrong}
Let \ct{C} be a locally small and cocomplete category, and $\lambda$ be a limit ordinal. If $\mathbb{A}$ is a small double category over \ct{C} such that $Aa$ is $\lambda$-small for each object $a$ in $\mathbb{A}$, then the algebraic weak factorization system cofibrantly generated by $\mathbb{A}$ exists.
\end{theo}
\begin{proof}
    As in the proof of \cite[Theorem 18]{bergetal25} we need to show that a certain category of special algebras is monadic over $\ct{C}^\to$. Since the latter is cocomplete as well, it suffices by \refcoro{monadicityspecialFalg} to show that the relevant functors $F$ and $G$ preserves colimits of shape $\Downarrow \lambda$. This follows from \cite[Proposition 11]{bergetal25} and the assumption that all $Aa$ are $\lambda$-small.
\end{proof}

\section{Conclusion}

We have shown that the standard account of ordinals in a constructive setting supports a useful theory of transfinite recursion, in that it can be used to give constructive proofs of  various initial algebra theorems and a strong form of Quillen's small object argument.

This paper also raises a number of questions that we have not answered and that we hope to take up in future work.
\begin{enumerate}
    \item A standard reference on transfinite arguments in the categorical setting is Kelly's paper \cite{Kelly1980A-unified}. Can we give a constructive account of this paper using ideas from this paper?
    \item More generally, there is the important theory of locally presentable and accessible categories \cite{Adamek1994Locally}. To which extend is it possible to develop this theory constructively?
    \item To which extent is it possible to prove the existence of initial algebras in categories that only have filtered colimits using a directed notion of ordinals, like the Tarski ordinals from \cite{joyalmoerdijk95}? This question is the topic of ongoing joint work with Giacomo De Antonellis.
    \item The results in this paper are most useful when there is a rich supply of limit ordinals, but we have left the question of their existence open. In particular, a natural question is what is needed to show that any set is $\lambda$-small for some limit ordinal $\lambda$. This will lead to some interesting metatheoretic considerations, that we plan to take up in a future paper.
\end{enumerate}

\bibliographystyle{plain}
\bibliography{ordinals}

\begin{thebibliography}{10}

\bibitem{aczelrathjen01}
P.~Aczel and M.~Rathjen.
\newblock Notes on constructive set theory.
\newblock Technical Report No. 40, Institut Mittag-Leffler, 2000/2001.

\bibitem{adamek74}
J.~Ad\'amek.
\newblock Free algebras and automata realizations in the language of categories.
\newblock {\em Comment. Math. Univ. Carolinae}, 15:589--602, 1974.

\bibitem{Adamek1994Locally}
J.~Ad\'amek and J.~Rosick\'y.
\newblock {\em Locally presentable and accessible categories}, volume 189 of {\em London Mathematical Society Lecture Note Series}.
\newblock Cambridge University Press, Cambridge, 1994.

\bibitem{bergetal25}
B.~van~den Berg, J.~Bourke, and P.~Seip.
\newblock A constructive approach to the double-categorical small object argument.
\newblock arXiv:2512.11692, 2025.

\bibitem{borceux94ii}
F.~Borceux.
\newblock {\em Handbook of categorical algebra. 2}, volume~51 of {\em Encyclopedia of Mathematics and its Applications}.
\newblock Cambridge University Press, Cambridge, 1994.
\newblock Categories and structures.

\bibitem{BourkeEquipping}
J.~Bourke.
\newblock Equipping weak equivalences with algebraic structure.
\newblock {\em Math. Z.}, 294(3-4):995--1019, 2020.

\bibitem{Bourke2016Accessible}
J.~Bourke and R.~Garner.
\newblock Algebraic weak factorisation systems {I}: {A}ccessible {AWFS}.
\newblock {\em J. Pure Appl. Algebra}, 220(1):108--147, 2016.

\bibitem{coquandetal23}
T.~Coquand, H.~Lombardi, and S.~Neuwirth.
\newblock Constructive theory of ordinals.
\newblock In M.~Benini; O. Beyersdorff; M. Rathjen;~P.M. Schuster, editor, {\em Mathematics for Computation - M4C}, pages 287--318. World Scientific, 2023.

\bibitem{Garner2011Understanding}
R.~Garner.
\newblock Understanding the small object argument.
\newblock {\em Appl. Categ. Structures}, 17(3):247--285, 2009.

\bibitem{dejongetal23}
T.~de Jong, N.~Kraus, F.~Nordvall~Forsberg, and C.~Xu.
\newblock Set-theoretic and type-theoretic ordinals coincide.
\newblock In {\em 38th Annual {ACM/IEEE} Symposium on Logic in Computer Science, {LICS} 2023, Boston, MA, USA, June 26-29, 2023}, pages 1--13. {IEEE}, 2023.

\bibitem{dejongetal25}
T.~de Jong, N.~Kraus, F.~Nordvall~Forsberg, and C.~Xu.
\newblock Ordinal exponentiation in homotopy type theory.
\newblock In {\em 40th Annual {ACM/IEEE} Symposium on Logic in Computer Science, {LICS} 2025, Singapore, June 23-26, 2025}, pages 262--274. {IEEE}, 2025.

\bibitem{joyalmoerdijk95}
A.~Joyal and I.~Moerdijk.
\newblock {\em Algebraic set theory}, volume 220 of {\em London Mathematical Society Lecture Note Series}.
\newblock Cambridge University Press, Cambridge, 1995.

\bibitem{Kelly1980A-unified}
G.~M. Kelly.
\newblock A unified treatment of transfinite constructions for free algebras, free monoids, colimits, associated sheaves, and so on.
\newblock {\em Bull. Austral. Math. Soc.}, 22(1):1--83, 1980.

\bibitem{Koubek1979Categorical}
V.~Koubek and J.~Reiterman.
\newblock Categorical constructions of free algebras, colimits, and completions of partial algebras.
\newblock {\em J. Pure Appl. Algebra}, 14(2):195--231, 1979.

\bibitem{krausetal21}
N.~Kraus, F.~Nordvall~Forsberg, and C.~Xu.
\newblock Connecting constructive notions of ordinals in homotopy type theory.
\newblock In F.~Bonchi and S.J. Puglisi, editors, {\em 46th International Symposium on Mathematical Foundations of Computer Science, {MFCS} 2021, Tallinn, Estonia, August 23-27, 2021}, LIPIcs, pages 70:1--70:16. Schloss Dagstuhl - Leibniz-Zentrum f{\"{u}}r Informatik, 2021.

\bibitem{martinlof84}
P.~Martin-L\"of.
\newblock {\em Intuitionistic type theory}, volume~1 of {\em Studies in Proof Theory. Lecture Notes}.
\newblock Bibliopolis, Naples, 1984.
\newblock Notes by G. Sambin.

\bibitem{pittssteenkamp21}
A.M. Pitts and S.~C. Steenkamp.
\newblock Constructing initial algebras using inflationary iteration.
\newblock In K.~Kishida, editor, {\em Proceedings of the Fourth International Conference on Applied Category Theory, {ACT} 2021, Cambridge, United Kingdom, 12-16th July 2021}, volume 372 of {\em {EPTCS}}, pages 88--102, 2021.

\bibitem{univalent13}
The Univalent~Foundations Program.
\newblock {\em Homotopy type theory---univalent foundations of mathematics}.
\newblock The Univalent Foundations Program, Princeton, NJ; Institute for Advanced Study (IAS), Princeton, NJ, 2013.

\bibitem{Quillen1967}
D.G. Quillen.
\newblock {\em Homotopical algebra}, volume No. 43 of {\em Lecture Notes in Mathematics}.
\newblock Springer-Verlag, Berlin-New York, 1967.

\bibitem{taylor96}
P.~Taylor.
\newblock Intuitionistic sets and ordinals.
\newblock {\em J. Symbolic Logic}, 61(3):705--744, 1996.

\end{thebibliography}

\appendix

\section{A comparision with Algebraic Set Theory}

In this appendix we will compare our theory of ordinals with the theory of ZF-algebras as developed by Joyal and Moerdijk in their book on \emph{Algebraic Set Theory} \cite{joyalmoerdijk95}. In particular, we claim that our theory of ordinals as given in Section 2 provides an alternative presentation of the initial ZF-algebra with a progressive successor (in \cite{joyalmoerdijk95} these ordinals are called the \emph{von Neumann ordinals}).

\begin{defi}{ZFalgebra} A \emph{ZF-algebra} is a partially ordered class $(X, \cleq)$ with small suprema additionally equipped with a unary (``successor'') operation $s: X \to X$. A morphism of ZF-algebras is a monotone map preserving small suprema and the successor operation. We will call the successor operation \emph{progressive} (or \emph{inflationary}) if $x \cleq sx$ holds for all $x \in X$.
\end{defi}

\begin{theo}{OrdinitZFalgwithinflationarysucc}
    If we equip ${\rm Ord}$ with the partial order
    \[ \alpha \cleq \beta :\Longleftrightarrow \Downarrow \alpha \subseteq \Downarrow \beta \]
    and the successor operation $s \alpha = \sup \{ \alpha \}$, then it becomes the initial ZF-algebra with a progressive successor.
\end{theo}
\begin{proof}
    For the operation $\bigcurlyvee$ defined before \reflemm{decompositionofsup}, we have $\Downarrow \bigcurlyvee A = \bigcup_{\alpha \in A} \Downarrow \alpha$, which implies that it computes the small suprema in the partial order $({\rm Ord}, \cleq)$ has small suprema given by $\bigcurlyvee$. The successor operator we have just defined is also clearly progressive.

    If $(X, \sqsubseteq, \sigma)$ is another ZF-algebra with an inflationary successor, then we can define by recursion a map $\varphi: {\rm Ord} \to X$ satisfying
    \[ \varphi(\alpha) := \bigsqcup_{\beta \in \Downarrow \alpha} \sigma(\varphi(\beta)). \]
    In fact, it follows from item (4) of \reflemm{decompositionofsup} that any morphism of ZF-algebras ${\rm Ord} \to X$ must be defined this way. So it remains to verify that $\varphi$ does indeed define a morphism of ZF-algebras.

    First of all, $\varphi$ is monotone (with respect to $\cleq$) by construction. 
    
    To see that it also preserves small suprema note that
    \[ \varphi(\bigcurlyvee A) = \bigsqcup_{\beta \in \Downarrow \bigcurlyvee A} \sigma(\varphi(\beta)) = \bigsqcup_{\beta \in \bigcup_{\alpha \in A} \Downarrow \alpha} \sigma(\varphi(\beta)) = \bigsqcup_{\alpha \in A} \bigsqcup_{\beta \in \Downarrow \alpha} \sigma(\varphi(\beta)) = \bigsqcup_{\alpha \in A} \varphi(\alpha). \]

    Finally, we need to check that it preserves the successor operation. To see this, note that $\sigma$ is progressive, and so
    \[ \varphi(s\alpha) = \bigsqcup_{\beta \in \Downarrow \alpha \cup \{ \alpha \}} \sigma(\varphi(\beta)) = \sigma(\varphi(\alpha)) \sqcup \bigsqcup_{\beta \in \Downarrow \alpha} \sigma(\varphi(\beta)) = \sigma(\varphi(\alpha)) \sqcup \varphi(\alpha) = \sigma(\varphi(\alpha)), \]
    which completes the proof.
\end{proof}

It may be useful to see the converse of this: if $(O, \cleq, s)$ is the initial ZF-algebra with a progressive successor $s$, then it is a model of our theory of ordinals. Let us first define:
\begin{eqnarray*} 
    \alpha \lt \beta & \Leftrightarrow_{def} & s\alpha \cleq \beta \\
    \alpha \leq \beta & \Leftrightarrow_{def} & \alpha \lt \beta \mbox{ or } \alpha = \beta \\
    \Downarrow \alpha & := & \{ \, \beta \, : \, \beta \lt \alpha \, \}
\end{eqnarray*}

From the fact that $(O, \cleq)$ is a partial order and $s$ is progressive it follows that:
\begin{lemm}{elementaryproperties}
    \begin{enumerate}
        \item[(i)] $\alpha \lt \beta \Longrightarrow \alpha \leq \beta \Longrightarrow \alpha \cleq \beta$.
        \item[(ii)] Both $\leq$ and $\lt$ are transitive. 
        \item[(iii)] $\alpha \lt \beta, \beta \cleq \beta' \Longrightarrow \alpha \lt \beta'$.
        \item[(iv)] $\cleq$ is a preorder. 
        \item[(v)] $\alpha \lt s\alpha$. 
        \item[(vi)] $\alpha \leq \beta \Longrightarrow s\alpha \cleq s\beta$.
    \end{enumerate}
\end{lemm}

The main step in the analysis of the Von Neumann ordinals is the following. (Compare Theorems 1.2 and 2.3 in \cite{joyalmoerdijk95}.)

\begin{prop}{downsetsinVNOrd} For the initial ZF-algebra with progressive successor $(O, \cleq, s)$ the following hold:
\begin{enumerate}
    \item $\alpha \lt \bigcurlyvee_{i \in I} \beta_i \Longleftrightarrow (\exists i \in I) \, \alpha \lt \beta_i$. Put differently: $\Downarrow \bigcurlyvee_i \beta_i = \bigcup_i \Downarrow \beta_i$.
    \item $\alpha \lt s\beta \Longleftrightarrow \alpha \leq \beta$. Put differently: $\Downarrow s\beta = \{ \beta \} \cup \Downarrow \beta $.
\end{enumerate}
\end{prop}
\begin{proof} Consider ${\rm Pow}(O)$, the collection of subsets of $O$, preordered as follows:
\[ E \leq F \Longleftrightarrow_{\rm def} (\forall \alpha \in E) \, (\exists \beta \in F) \, \alpha \leq \beta. \]
This preorder has suprema given by unions and comes equipped with a progressive successor operation 
\[ \sigma(E) = \{ \bigcurlyvee_{\alpha \in E} s\alpha \}. \]
By taking its poset reflection, we obtain another ZF-algebra $(P, \leq, \sigma)$ with progressive successor. (Note that something is needed to argue that it still has suprema: in a set theory like {\bf CZF} we would have to use the Collection Axiom.)

It is not hard to see that $r: P \to O$ defined by
\[ r(E) := \bigcurlyvee_{\alpha \in E} s\alpha \]
is a morphism of ZF-algebras. Therefore the initiality of $O$ gives us a morphism of ZF-algebras $i: O \to P$ such that $r \circ i = 1_O$. In other words, we have that:
\[ \alpha = \bigcurlyvee_{\beta \in i(\alpha)} s\beta. \]

From this it follows that
\[ is(\alpha) = \sigma(i(\alpha)) = \{ \bigvee_{\beta \in i(\alpha)} s\beta \} = \{ \alpha \}, \]
and therefore that
\[ ir(E) = i(\bigcurlyvee_{\alpha \in E} s\alpha) = \bigvee_{\alpha \in E} is\alpha = \bigvee_{\alpha \in E} \{ \alpha \} = E. \]
We conclude that $r$ is an isomorphism.

For the first property, we note that it holds in $P$ and therefore in $O$ as well. Indeed, $\alpha \lt \bigcurlyvee_{i \in I} \beta_i$ is equivalent to $s\alpha \cleq \bigcurlyvee_{i \in I} \beta_i$ and in $P$ the successor is a singleton and the supremum a union.

For the second property, we observe that
\[ \beta \lt s(\alpha) \Rightarrow s(\beta) \cleq s(\alpha) \Rightarrow is(\beta) \leq is(\alpha) \Rightarrow \{ \beta \} \leq \{ \alpha \} \Rightarrow \beta \leq \alpha. \]
This completes the proof.
\end{proof}

The initiality of $(O, \cleq, s)$ also gives us following induction principle:
\begin{equation} \label{indforO} 
    (\forall \alpha) \, \big( \, \varphi(\alpha) \to \varphi(s\alpha) \, \big) \to (\forall \alpha_i) \big( \, (\forall i \in I) \, \varphi(\alpha_i) \to \varphi(\bigvee_i \alpha_i) \, \big) \to (\forall \alpha) \, \varphi(\alpha).
\end{equation}
The next lemma lists some consequences of this.

\begin{lemm}{conseqofinduction}
\begin{enumerate}
    \item[(i)] $\big( \, (\forall \alpha) \, ( \, (\forall \beta \lt \alpha) \, \psi(\beta)  \to \psi(\alpha) \, ) \, \big) \to (\forall \alpha) \, \psi(\alpha)$.
    \item[(ii)] $\Downarrow \alpha$ is a set.
    \item[(iii)] $\alpha = \bigvee_{\beta \in \Downarrow \alpha} s\beta$.
    \item[(iv)] $\Downarrow \alpha = \Downarrow \beta$ implies $\alpha = \beta$.
\end{enumerate}
\end{lemm}
\begin{proof} (i) Prove $(\forall \beta \lt\alpha) \, \psi(\alpha)$ by the induction principle (\ref{indforO}).

Items (ii) and (iii) are also proved by the induction principle (\ref{indforO}) using the previous proposition.

Item (iv) is a direct consequence of (iii).
\end{proof}

\begin{theo}{VNordmodelofOrd}
    Suppose $(O, \cleq, s)$ is the initial ZF-algebra with progressive successor. With a sup defined as
    \[ \sup: {\rm Pow}({\rm Ord}) \to {\rm Ord}: A \mapsto \bigcurlyvee_{\alpha \in A} s\alpha, \]
    it becomes a model of our theory of ordinals. 
\end{theo}
\begin{proof}
    We have
    \begin{eqnarray*} \beta \lt \sup A & \Leftrightarrow & \beta \lt \bigcurlyvee_{\alpha \in A} s\alpha \\ & \Leftrightarrow & (\exists \alpha \in A) \, \beta \lt s\alpha \\ & \Leftrightarrow & (\exists \alpha \in A) \, \beta \leq \alpha \\ & \Leftrightarrow  & \beta \in A \lor (\exists \alpha \in A) \, \beta \lt \alpha. \end{eqnarray*}
    Therefore it remains to check that $O$ satisfies the induction principle:
    \begin{displaymath}
        \begin{array}{cc}
        ({\rm Induction}) & \frac{(\forall A) \, \big( \, (\forall \alpha \in A) \, \varphi(\alpha)  \to \varphi(\sup A) \, \big)}{(\forall \alpha) \, \varphi(\alpha) }
        \end{array}
    \end{displaymath}
    This will follow from the induction principle (i) in \reflemm{conseqofinduction} once we prove that:
    \begin{enumerate}
        \item[(i)] $s\alpha = \sup(\{ \alpha \})$
        \item[(ii)] $\bigcurlyvee A = \sup(\bigcup_{\alpha \in A} \Downarrow \alpha)$
    \end{enumerate}
    For this we note that we have
    \[ \beta \lt \sup(\{ \alpha \})  \Leftrightarrow \beta = \alpha \lor \beta \lt \alpha \Leftrightarrow \beta \lt s\alpha, \]
    and therefore $s\alpha = \sup(\{ \alpha \})$ by the previous proposition and item (iv) from \reflemm{conseqofinduction}.

    Finally, we also have
    \[ \beta \lt \sup(\bigcup_{\alpha \in A} \Downarrow \alpha) \Leftrightarrow (\exists \alpha \in A) \, (\exists \gamma \lt \alpha) \, \beta \leq \gamma \Leftrightarrow (\exists \alpha \in A) \, \beta \lt \alpha, \]
    and therefore $\bigcurlyvee A = \sup(\bigcup_{\alpha \in A} \Downarrow \alpha)$ by the previous proposition and item (iv) from \reflemm{conseqofinduction}.
\end{proof}

\section{Criteria for monadicity}

\begin{defi}{splitcoequalizer}
    A diagram of the form
    \begin{displaymath}
        \begin{tikzcd}
            Z \ar[r, shift left, "f"] \ar[r, shift right, "g"'] & Y \ar[r, "h"] & X
        \end{tikzcd}
    \end{displaymath}
    is a \emph{split coequalizer} if $h \circ f = h \circ g$ and $h$ and $f$ have sections $s$ and $t$ such that $g \circ t = s \circ h$.
\end{defi}

\begin{theo}{checkformonadicity}
    Let $F$ be an endofunctor on a cocomplete category \ct{C}. If \ct{D} is the category of $F$-algebras and $U: \ct{D} \to \ct{C}$ the forgetful functor, then:
    \begin{enumerate}
        \item[(i)] $U$ is conservative (reflects isomorphisms).
        \item[(ii)] Any parallel pair of maps $f, g: Z \to Y$ in \ct{D} such that $Uf, Ug: UZ \to UY$ has a split coequalizer in \ct{C} has a coequalizer $h: Y \to X$ in \ct{D} preserved by $U$. 
    \end{enumerate}
    The same statements hold if $U$ is the forgetful functor and \ct{D} is the category of algebras for a pointed or special endofunctor.
\end{theo}
\begin{proof}
    It is easy to see that the forgetful functor reflects isomorphisms in the case of $F$-algebras, which also implies that it doesthe seem in the case of pointed or special algebras, as these form full subcategories.

    So let $(Z, z: FZ \to Z)$ and $(Y, y: FY \to Y)$ be $F$-algebras and $f, g: (Z, z) \to (Y, y)$ be morphisms of $F$-algebras. If there are maps $h: Y \to X, s: Y \to X, t: Y \to Z$ such that
    \[ h \circ f = h \circ g, \quad h \circ s = 1_X, \quad f \circ t = 1_Y, \quad \mbox{and} \quad g \circ t = s \circ h, \]
    then we can turn $X$ into an $F$-algebra $(X, x: FX \to X)$ by defining
    \[ x:=  h \circ y \circ Fs. \]
    This turns $h$ into a morphism of $F$-algebras, as $x \circ Fh = h \circ y \circ Fs \circ Fh = h \circ y$.

    To see that $h$ is a coequalizer in the category of $F$-algebras, let us assume that $(V, v: TV \to V)$ is an $F$-algebra and $k: (Y, y) \to (V, v)$ is a morphism of $F$-algebras such that $k \circ f = k \circ g$. Put $l := k \circ s$. 
    
    We claim that $l$ is a morphism of $F$-algebras $(X, x) \to (V, v)$ such that $l \circ h = k$. Indeed, we have
    \[ l \circ h = k \circ s \circ h = k \circ g \circ t = k \circ f \circ t = k, \]
    and 
    \[ l \circ x = l \circ h \circ y \circ Fs = k \circ y \circ Fs = v \circ Fk \circ Fs = v \circ Fl. \]
    The morphism $l$ is necessarily unique with these properties as $h$ is a split epi in \ct{C}.

    It remains that verify that $(X,x)$ will be pointed or special whenever $(Y, y)$ is. If $(Y, y)$ is pointed, then the following calculation
    \[ x \circ \eta_X = h \circ y \circ Fs \circ \eta_X = h \circ y \circ \eta_Y \circ s = h \circ s = 1_X \]
    shows that $(X, x)$ is pointed as well. And if $(Y, y)$ is special, then we have
    \begin{eqnarray*}
        x \circ Fx \circ \sigma_X & = & x \circ Fh \circ Fy \circ FFs \circ \sigma_X \\
        & = & h \circ y \circ Fy \circ \sigma_Y \circ Gs \\
        & = & h \circ y \circ \tau_Y \circ Gs \\
        & = & h \circ y \circ Fs \circ \tau_X \\
        & = & x \circ \tau_X
    \end{eqnarray*}
    showing that $(X, x)$ is special.
\end{proof}
\end{document}